\documentclass{aims}
\usepackage{amsmath}
\usepackage{paralist}
\usepackage{graphics} 
\usepackage{epsfig} 
\usepackage{graphicx}  \usepackage{epstopdf}
\usepackage[colorlinks=true]{hyperref}
\hypersetup{urlcolor=blue, citecolor=red}

\def\currentvolume{X}
 \def\currentissue{X}
  \def\currentyear{200X}
   \def\currentmonth{XX}
    \def\ppages{X--XX}
     \def\DOI{10.3934/xx.xx.xx.xx}

\newtheorem{theorem}{Theorem}[section]

\newtheorem{proposition}{Proposition}

\theoremstyle{definition}
\newtheorem{definition}[theorem]{Definition}

\newcommand{\udm}{u_{d,m}}
\newcommand{\dep}{d_{\varepsilon}}
\newcommand{\mep}{m_{\varepsilon}}
\newcommand{\uep}{u_\varepsilon}
\newcommand{\vep}{v_\varepsilon}
\newcommand{\udepmep}{u_{d_\varepsilon, m_\varepsilon}}
\newcommand{\R}{\mathbb{R}}

\newcommand{\e}{\varepsilon}
\newcommand{\sfrac}[2]{\text{\small $\dfrac{#1}{#2}$}}

\title[RATIO OF SPECIES AND RESOURCES] 
{On the unboundedness of the ratio of species and resources for the diffusive logistic equation}

\author[Jumpei Inoue and Kousuke Kuto]{}


\email{j-inoue@uec.ac.jp}
\email{kuto@waseda.jp}

\thanks{The second author is supported by JSPS KAKENHI
Grant-in-Aid Grant Number 19K03581}

\thanks{$^*$ Corresponding author: Jumpei Inoue}

\begin{document}
\maketitle

\centerline{\scshape Jumpei Inoue$^*$}
\medskip
{\footnotesize
 \centerline{Department of Graduate School of Informatics and Engineering}
   \centerline{The University of Electro-Communications}
   \centerline{1-5-1 Chofugaoka, Chofu, Tokyo, 182-8585, Japan}
} 

\medskip

\centerline{\scshape Kousuke Kuto}
\medskip
{\footnotesize
 \centerline{Department of Applied Mathematics}
   \centerline{Waseda University}
   \centerline{3-4-1 Ohkubo, Shinjuku-ku, Tokyo, 164-8555, Japan}
}

\bigskip

\centerline{(Communicated by the associate editor name)}

\begin{abstract}

Concerning a class of diffusive logistic equations, Ni \cite[Abstract]{BaiHeLi} proposed an optimization problem to consider the supremum of the ratio of the $L^1$ norms of species and resources by varying the diffusion rates and the profiles of resources, and moreover, he gave a conjecture that the supremum is $3$ in the one-dimensional case.
In \cite{BaiHeLi}, Bai, He and Li proved the validity of this conjecture.
The present paper shows that the supremum is infinity in a case when the habitat is a multi-dimensional ball.
Our proof is based on the sub-super solution method.
A key idea of the proof is to construct an $L^1$ unbounded sequence of sub-solutions.

\end{abstract}

\section{Introduction}

This paper is concerned with the following stationary problem for a diffusive logistic equation
\begin{equation}\label{main_eq}
  \begin{cases}
    d\, \Delta u + u(m(x)-u)=0 &\text{in}\;\;\Omega, \\
    \; \partial_\nu u = 0 &\text{on}\;\;\partial\Omega,
  \end{cases}
\end{equation}
where $\Omega \subseteq \R^n$ is a bounded domain with a smooth boundary $\partial\Omega$; $d$ is a positive constant; $m(x)$ is a measurable function belonging to
\[ L_+^\infty (\Omega) := \{\, f \in L^\infty(\Omega) \mid f(x) \ge 0 \;\text{a.e.}\; x \in \Omega,\; \|f\|_{L^\infty} > 0 \,\}. \]
From the viewpoint of an ecological model, \eqref{main_eq} is expected to realize the stationary distribution of species in the habitat $\Omega$.
In this sense, the unknown function $u(x)$ represents the distribution of species and $m(x)$ can be interpreted as the distribution of resources (feed).
The boundary condition assumes that there is no flux of species on the boundary $\partial\Omega$ of the habitat.
In the field of reaction-diffusion equations, the existence, uniqueness and stability of positive solutions are obtained by Cantrell and Cosner \cite{CantrellCosner}.
Besides \cite{CantrellCosner},
series of works by them Cantrell and Cosner \cite{CantrellCosner2, CantrellCosner3, CantrellCosner4},
Taira \cite{Taira1, Taira2} gave a great contribution to the research field for a
class of diffusive logistic equations with spatial heterogeneous terms.
\begin{proposition}[\cite{CantrellCosner}]\label{prop1}
  For each $d > 0$ and each $m \in L_+^\infty(\Omega)$, \eqref{main_eq} has a unique positive solution $\udm(x)$ in the class of $W^{2,p}(\Omega)$ for any $p \ge 1$. Furthermore, $\udm(x)$ is globally asymptotically stable (GAS) in the sense that it attracts all positive solutions of the corresponding parabolic problem as $t \to \infty$.
\end{proposition}

In the sense of Proposition \ref{prop1}, one can say that, in order to know or design the final state of the distribution of species, it is important to derive mathematical effects of the diffusion rate $d$ and the distribution $m(x)$ of resources on the profile of $\udm(x)$.
As a trigger to know such effects, the following mathematical procedure for \eqref{main_eq} was introduced by Lou \cite{Lou1} (see also Ni \cite{Ni} and references therein):
Dividing the first equation of \eqref{main_eq} by $u(x)$ and integrating the resulting expression gives
\[ d \int_\Omega \frac{\Delta u}{u} = \|u\|_{L^1(\Omega)} - \|m\|_{L^1(\Omega)}. \]
By the boundary condition, integration by parts in the left-hand side leads to
\[ \|u\|_{L^1(\Omega)} - \|m\|_{L^1(\Omega)} = d \int_\Omega \biggl( \frac{|\nabla u|}{u} \biggr)^2 \ge 0. \]
Then one can see that
\begin{equation}\label{ratio}
  \frac{\|\udm\|_{L^1(\Omega)}}{\|m\|_{L^1(\Omega)}} \ge 1 \quad\text{for any}\;\; (d,m) \in (0,\infty) \times L^\infty_+(\Omega),
\end{equation}
where the equality holds only when $m(x)$ and $\udm(x)$ identically equal to a positive constant $m_0$.
In the ecological sense, we can regard $\|\udm\|_{L^1(\Omega)}$ and $\|m\|_{L^1(\Omega)}$ as the total population of species and the total amount of resources, respectively.
Then \eqref{ratio} means that the heterogeneity of resource can benefit species. 
Motivated by this fact, some optimization problems concerning \eqref{main_eq} have been studied in the field of elliptic equations.
We refer to \cite{LiLou, LiangLou, Lou1} and \cite{DeAngelisZhangNiWang, Mazzari, MazzariNadinPrivat, NagaharaYanagida} for the dependence of $u_{d,m}$ upon $d>0$
(for fixed $m$) and $m$ (for fixed $d$), respectively.
See \cite{HeLamLouNi, HeNi1, HeNi2, HeNi3, HeNi4, HeNi5, Lou3, LouWang}
for applications of information on $u_{d,m}$ to the dynamics of solutions
to a class of diffusive Lotka-Volterra systems.
We also refer to book chapters
\cite{LamLou}, \cite{Lou2} and \cite{Ni} to know trends of studies for
\eqref{main_eq} and related problems.

This paper focuses on a biological question: ``How can we maximize the total population under the limited total resources?''
From such a viewpoint, Ni proposed the following optimization problem: ``What is the supremum of
\begin{equation}\label{rate}
  \frac{\|\udm\|_{L^1(\Omega)}}{\|m\|_{L^1(\Omega)}}
\end{equation}
for any $d>0$ and any $m \in L_+^\infty(\Omega)$?'', and moreover, he gave a conjecture that the supremum is $3$ in the one-dimensional case when $\Omega=(0,\ell)$ (see \cite[Abstract]{BaiHeLi}).
Concerning this conjecure, Bai, He and Li \cite{BaiHeLi} proved the validity.
The procedure of their proof \cite{BaiHeLi} first shows that $\|\udm\|_{L^1(0,\ell)} < 3 \|m\|_{L^1(0,\ell)}$ for any $(d,m) \in (0,\infty) \times L_+^\infty(0,\ell)$, and next, shows that for
\begin{equation}\label{BHL}
  \dep = \sqrt{\e},\quad
  \mep(x) = \begin{cases} 
    1/\e &\text{for}\;\; x \in [0,\e],\\
    \;0 &\text{for}\;\; x \in (\e, \ell] 
  \end{cases}
\end{equation}
with small $\e > 0$, the solution $\uep(x) = \udepmep(x)$ of \eqref{main_eq} with $\Omega = (0,\ell)$ satisfies
\begin{equation}\label{to3}
  \frac{\|\uep\|_{L^1(0,\ell)}}{\|\mep\|_{L^1(0,\ell)}} = \|\uep\|_{L^1(0,\ell)} \nearrow 3 \quad\text{as}\quad \e \searrow 0.
\end{equation}
It can be verified that $\uep(x)$ is monotone decreasing for $x \in (0,\ell)$ and decays to zero over any compact set contained in $(0,\ell]$ as $\e \to 0$, but $\uep(0)$ blows up as $\e \to 0$.
Here it should be noted that their elegant proof using the energy method established \eqref{to3} without any more detailed profiles of $\uep(x)$.
In \cite{Inoue}, the first author of the present paper derived some detailed information on the profile of $\uep(x)$.
Among other things, he obtained
\begin{equation}\label{3/2}
  \lim_{\e \to 0} \sqrt{\e} \uep(x) = \frac{3}{2} \quad(0 \leq x \leq \e).
\end{equation}
In the one-dimensional habitat case, \eqref{BHL} tells that a concentration of resources and a suitable small diffusion rate make the total population per the total resources be a maximizing sequence.
Furthermore, \eqref{3/2} means that, in the resource interval $[0,\e]$, the growth rate $O(1/\sqrt{\e})$ of species is less than that of resource.

This paper considers the supremum of the ratio in \eqref{ratio} in the case when $\Omega$ is a unit ball $B_1^n := \{\, x \in \R^n \mid |x|<1\, \}$.
The following theorem is a crucial part of a main result (Theorem \ref{mainthm2}):
\begin{theorem}\label{main_thm}
  Let $\udm(x)$ be a positive solution of \eqref{main_eq} with $\Omega = B_1^n$. If $n \ge 2$, then
  \begin{equation}\label{L1ratio_infty}
    \sup_{(d,m) \in (0,\infty) \times L_+^\infty(B_1^n)} \frac{\|\udm\|_{L^1(B_1^n)}}{\|m\|_{L^1(B_1^n)}} = \infty.
  \end{equation}
\end{theorem}
This result is a big contrast to that of the one-dimensional case (\cite{BaiHeLi}) where the above supremun is $3$,
and moreover, gives a negative answer to an open question in
\cite[(8.36)]{LamLou}.
The proof of Theorem \ref{main_thm} is based on the sub-super solution method.
We employ a concentration setting of resources near the center as $\mep(x) = 1/\e^n$ for $x \in \overline{B_\e^n} := \{\, x \in \R^n \mid |x| \le \e \,\}$ and $\mep(x) = 0$ otherwise. 
Then a control of the diffusion rate as $\dep = O(1 / \e^{n-2})$ enables us to construct an $L^1$ unbounded sequence of sub-solutions as $\e \to +0$.
This sub-solution also ensures that the growth rate of species in the resource region $\overline{B_\e^n}$ is equal to $O(1/\e^n)$ which is same as that of resources.

This paper consists of three sections.
Section 2 is devoted to the proof of the main result.
In Section 3, some concluding remarks related to the result will be given.

\section{Construction of an \texorpdfstring{$L^1$}{}-unbounded sequence of solutions}

The proof of Theorem \ref{main_thm} is based on the (weak) sub-super solution method for a class of elliptic equations. Since $m(x)$ is allowed to be a discontinuous function, we note a framework of the method. Consider the following Neumann problem for a class of semilinear elliptic equations including \eqref{main_eq}:
\begin{equation}\label{semilinear_Nuemann}
  \begin{cases}
    d\, \Delta u + f(x,u) = 0 &\text{in}\;\; \Omega,\\
    \; \partial_{\nu}u = 0 &\text{on}\;\; \partial\Omega,
  \end{cases}
\end{equation}
where $f(x,t)$ is a Carath\'eodory function for $(x,t) \in \Omega \times \R$, that is, for any fixed $t \in \R$, $x \mapsto f(x,t)$ is a measurable function in $\Omega$ and for any fixed $x \in \Omega$, $t \mapsto f(x,t)$ is a continuous function.

\begin{definition}(e.g., \cite[p.52]{Du})
  A function $u(x)$ is called a (weak) sub-solution of \eqref{semilinear_Nuemann} if $u \in W^{1,p}(\Omega) \;(p>1)$, $f(x,u(x))$ belongs to $L^{p/(p-1)}(\Omega)$ and
  \begin{equation}\label{var}
    d \int_{\Omega} \nabla u \cdot \nabla \varphi \le \int_{\Omega} f(x,u(x)) \varphi
  \end{equation}
  for any $\varphi \in W^{1,p}(\Omega)$ with $\varphi \ge 0$ a.e. in $\Omega$. If the inequality in \eqref{var} is reversed, $u(x)$ is called a (weak) super-solution.
\end{definition}
The following proposition is fundamental but useful and will play an important role in the proof of Theorem \ref{main_thm}.
\begin{proposition}(e.g., \cite[Theorems 4.9 and 4.12]{Du})\label{monoprop}
  Let $\underline{u}(x)$ and $\overline{u}(x)$ be (weak) sub- and super-solutions of \eqref{semilinear_Nuemann}, respectively, satisfying $\underline{u} \leq \overline{u}$ a.e. in $\Omega$. Suppose that there exists a function $k \in L^{p/(p-1)}(\Omega) \;(p>1)$ such that
  \[ |f(x,t)| \leq k(x) \;\;\text{for a.e.}\; x \in \Omega \;\;\text{and all}\;\; t \in [\underline{u}(x), \overline{u}(x)]. \]
  Then \eqref{semilinear_Nuemann} admits a weak solution $u(x)$ satisfying $\underline{u} \leq u \leq \overline{u}$ a.e. in $\Omega$.
\end{proposition}

Hereafter we consider \eqref{main_eq} in the case when $\Omega$ is the multi-dimensional unit ball $B^n_1=\{\, x \in \R^n \mid |x|<1 \,\}$ with $n \ge 2$. By referring the setting of $m(x)$ in \cite{BaiHeLi} for the one-dimensional case, we set
\begin{equation}\label{mep}
  m(x) = \mep(x) =
  \begin{cases}
    1 / \e^n &\text{for}\; x \in \overline{B^n_\e}, \\
    0 &\text{for}\; x \in \overline{B^n_1} \setminus \overline{B^n_\e}
  \end{cases} 
\end{equation}
for any $0 < \e < 1$. From the viewpoint of the ecological model, the above setting of $\mep(x)$ concentrates all resources near the center of the unit ball habitat. This location of resources differs from that in the one-dimensional case where all resources are put near an endpoint of the $(0,\ell)$. Hence it follows that
\[ \|\mep\|_{L^1(B^n_1)}=|B^n_1|, \]
where $|B^n_1|$ denotes the volume of $B^n_1$. The following theorem is a main result of this paper which immediately leads to Theorem \ref{main_thm}.
\begin{theorem}\label{mainthm2}
  Suppose that the dimension number $n$ satisfies $n \ge 2$. Then there exist positive constants $c_1$ and $c_2$ depending only on $n$ such that the unique positive solution $\uep(x)$ of
  \begin{equation}\label{inoue-eq}
    \begin{cases}
      \sfrac{c_1}{\e^{n-2}} \Delta u + u (\mep(x)-u) = 0 &\text{in}\;\; B^n_1,\\
      \; \partial_\nu u = 0 &\text{on}\;\; \partial B^n_1.
    \end{cases}
  \end{equation}
  satisfies
  \[ \frac{\|\uep\|_{L^1(B^n_1)}}{\|\mep\|_{L^1(B^n_1)}} \ge 
    c_2 \biggl( 1 - \frac{1}{e}  + \frac{n}{e}\,|\log \e| \biggr) \]
  for any $0 < \e < 1$.
\end{theorem}

\begin{proof}
It follows from Proposition \ref{prop1} that for each
\begin{equation}\label{dep}
  d = \dep := \frac{c_1}{\e^{n-2}}
\end{equation}
and $\mep(x)$ introduced by \eqref{mep}, there exists a unique positive solution $\uep(x)$ of \eqref{inoue-eq}. By virtue of Proposition \ref{monoprop}, if we can find a super-solution $\overline{\uep}(x)$ and a sub-solution $\underline{\uep}(x)$ satisfying
\begin{equation}\label{sub<super}
  0 < \underline{\uep} \leq \overline{\uep} \quad\text{in}\;\; B^n_1,
\end{equation}
then $\underline{\uep} \leq \uep \leq \overline{\uep}$ in\ $B^n_1$. Since $\|\mep\|_{L^1(B^n_1)} = |B^n_1|$ is independent of $0 < \e < 1$, then our strategy is to construct a sub-solution $\underline{\uep}(x)$ and a super-solution $\overline{u_\e}(x)$ satisfying not only \eqref{sub<super} but also
\[ \lim_{\e \to 0} \|\underline{\uep}\|_{L^1(B^n_1)} \to \infty. \]
To do so, we introduce two functions $\overline{\uep}(x)$ and $\underline{\uep}(x)$ defined over $\overline{B^n_1}$ as 
\begin{equation}\label{super_sol}
  \overline{\uep}(x) := \frac{1}{\e^{n}} \quad \text{for}\;\; x \in \overline{B^n_1}
\end{equation}
and
\begin{equation}\label{sub_sol}
  \underline{\uep}(x) :=
  \begin{cases}
    \; \sfrac{c_2}{\e^n} e^{-|x|^n / \e^n} &\text{for}\;\; x \in \overline{B^n_\e}, \vspace{1mm}\\
    \; \sfrac{c_2}{e|x|^n} &\text{for}\;\; x \in \overline{B^n_1} \setminus \overline{B^n_\e}.
  \end{cases}
\end{equation}
Here $c_2$ will be determined later independently of $0<\e<1$. It is easily verified that $\underline{\uep}(x)$ is in the $C^2$ class except for $|x| = \e$, but still in the $C^{1}$ class. In what follows, we seek for a range of parameters $(c_1,c_2)$ so that 
\begin{enumerate}
  \item[(a)] $\overline{\uep}(x)$ is a super-solution of \eqref{main_eq},
  \item[(b)] $\underline{\uep}(x)$ is a sub-solution of that, and
  \item[(c)] $(0 <)\, \underline{\uep} \leq \overline{\uep} \quad\text{in}\;\; B^n_1$.
\end{enumerate}

Since $\mep(x)$ is defined as \eqref{mep}, then $\overline{\uep}(x) \equiv 1/\e^n$ satisfies
\[ \dep \Delta \overline{\uep} + \overline{\uep} (\mep(x) - \overline{\uep}) =
  \begin{cases}
    \; 0 &\text{for}\;\; x \in \overline{B^n_\e}, \\
    - \overline{\uep}^{2} < 0 &\text{for}\;\; x \in B^n_1 \setminus \overline{B^n_\e}
  \end{cases} \]
and $\partial_\nu \overline{\uep} = 0$ on $\partial{B^n_1}$. Hence $\overline{\uep}$ is a super-solution of \eqref{inoue-eq}

Concerning (b), we have to check the inequality of \eqref{var}:
\begin{equation}\label{sub-condition}
  \dep \int_{B^n_1} \nabla \underline{\uep} \cdot \nabla \varphi \leq \int_{B^n_1} \underline{\uep} (\mep - \underline{\uep}) \varphi
\end{equation}
for any $\varphi \in W^{1,p}(B^n_1)$ with $\varphi \ge 0$ a.e. in $B^n_1$. Thanks to the fact
\[ \underline{\uep} \in C^{2}(\overline{B^n_1} \setminus \{ |x|=\e \}) \cap C^1(\overline{B^n_1}), \]
for the verification of \eqref{sub-condition}, it suffices to show that
\begin{equation}\label{diff_ineq1}
  \dep \Delta \underline{\uep} + \underline{\uep} (\mep - \underline{\uep}) \ge 0 \quad \text{for}\;\; x \in B^n_1 \setminus \{ |x|=\e \}
\end{equation}
and
\begin{equation}\label{diff_ineq2}
  \partial_\nu \underline{\uep} \leq 0 \quad\text{on}\;\; \partial B^n_1.
\end{equation}
The boundary condition \eqref{diff_ineq2} is obviously satisfied. Then our crucial task is to find a parameter range of $(c_1, c_2)$ satisfying \eqref{diff_ineq1}. Since $\underline{\uep}(x)$ is a radial function, we know that the required inequality \eqref{diff_ineq1} is equivalent to
\begin{equation}\label{diff_ineq1_radial}
  \begin{cases}
    \dep \biggl( \vep'' + \sfrac{n-1}{r} \vep' \biggr) + \vep (\tilde{m}_{\e}(r) - \vep) \ge 0 &\text{for}\;\; 0<r<1 \;\;\text{and}\;\; r \neq \e,\\
    \; \vep'(0) = 0, &
  \end{cases}
\end{equation}
where $\vep(r) := \underline{\uep}(x)$ and $\tilde{m}_{\e}(r) := \mep(x)$ for $r=|x| \in [0,1]$, that is,
\begin{equation}\label{vm_radial} 
  \vep(r) =
  \begin{cases}
    \; \sfrac{c_2}{\e^n} e^{-r^n / \e^n} &(0 \leq r \leq \e), \vspace{1mm} \\
    \; \sfrac{c_2}{e r^n} &(\e < r \leq 1),
  \end{cases} \qquad
  \tilde{m}_{\e}(r) =
  \begin{cases}
    1/\e^n & (0 \leq r \leq \e), \\
    \; 0 & (\e < r \leq 1),
  \end{cases} 
\end{equation}
and the prime symbol $'$ represents the derivative by $r$. Then straightforward calculations yield
\[ \vep'(r) = \begin{cases}
  - \sfrac{c_2 n r^{n-1}}{\e^{2n}} e^{-r^n / \e^n} &(0 \leq r \leq \e), \vspace{1mm}\\
  \; - \sfrac{c_2 n}{er^{n+1}} &(\e < r \leq 1)
  \end{cases} \]
and
\[ \vep''(r) = \begin{cases}
  \dfrac{c_2 n(n-1) r^{n-2}}{\e^{2n}} \biggl( \dfrac{n r^n}{(n-1)\e^n}-1 \biggr) e^{-r^n / \e^n} &(0 \leq r \leq \e),
  \vspace{1mm}\\
  \; \dfrac{c_2 n(n+1)}{e r^{n+2}} &(\e < r \leq 1).
  \end{cases} \]
Here it should be noted that $\vep \in C^2([0,\e) \cup (\e,1]) \cap C^1([0,1])$ and the multi-dimensional situation $n \ge 2$ ensures $\vep'(0) = 0$. For $0 < r < \e$, one can see
\begin{align*}
  \frac{c_1}{\e^{n-2}} &\left( \vep'' + \frac{n-1}{r} \vep' \right) + \vep \left(\frac{1}{\e^n} - \vep\right) \\
  &= e^{-r^n/\e^n} \left( \frac{c_1c_2n^2}{\e^{4n-2}}r^{2n-2} - \frac{2c_1c_2n(n-1)}{\e^{3n-2}}r^{n-2} + \frac{c_2}{\e^{2n}} - \frac{c_2^2}{\e^{2n}}e^{-r^n/\e^n} \right).
\end{align*}
To assure the positive minimum of the right-hand side, we estimate the bracket part as follows
\begin{align*}
  & \frac{c_1c_2n^2}{\e^{4n-2}}r^{2n-2} - \frac{2c_1c_2n(n-1)}{\e^{3n-2}}r^{n-2} + \frac{c_2}{\e^{2n}} - \frac{c_2^2}{\e^{2n}}e^{-r^n/\e^n} \\
  & > - \frac{2c_1c_2n(n-1)}{\e^{2n}} + \frac{c_2}{\e^{2n}} - \frac{c_2^2}{\e^{2n}} \\
  & = \frac{c_2}{\e^{2n}} (1-2c_1n(n-1)-c_2) \quad \text{for any}\;\; 0 < r < \e.
\end{align*}
Thus if $1-2c_1n(n-1)-c_2 \ge 0$, then the differential inequality \eqref{diff_ineq1_radial} holds for $0<r<\e$. On the other hand, for $\e < r < 1$, we know
\[ \frac{c_1}{\e^{n-2}} \left( \vep'' + \frac{n-1}{r} \vep' \right) - v_{\e}^2 = \frac{c_2}{er^{n+2}} \left( \frac{2c_1n}{\e^{n-2}} - \frac{c_2}{er^{n-2}} \right) \]
and
\[ \frac{2c_1n}{\e^{n-2}} - \frac{c_2}{er^{n-2}} \ge \frac{1}{\e^{n-2}} \left( 2c_1n - \frac{c_2}{e} \right). \]
Thus if $2c_1n - c_2/e \ge 0$, then \eqref{diff_ineq1_radial} holds for $\e < r < 1$. Therefore, we know that if $(c_1,c_2)$ satisfies
\begin{equation}\label{param}
  1-2c_1n(n-1)-c_2 \ge 0 \;\;\text{and}\;\; 2c_1n - \frac{c_2}{e} \ge 0,
\end{equation}
then \eqref{diff_ineq1} holds, and thereby, the required (b) is satisfied. Here it is noted that the set
\[ T:= \{\, (c_1, c_2) \in \R^{2}_{> 0} \mid (c_1,c_2) \text{\ satisfies \eqref{param}} \,\} \]
forms a triangle whose vertices are
\[ (c_1,c_2) = (0,0),\, \biggl(\sfrac{1}{2n(e+n-1)}, \sfrac{e}{e+n-1}\biggr),\, \biggl(\sfrac{1}{2n(n-1)}, 0\biggr). \]

The final condition (c) $\underline{\uep} \leq \overline{\uep}$ in $B^n_1$ holds true if and only if $c_2 \leq 1$. However the condition $c_2 \leq 1$ is already necessary for \eqref{param}.

Consequently, we can deduce that if $(c_1, c_2) \in T$, then $\overline{\uep}(x)$ and $\underline{\uep}(x)$ introduced by \eqref{super_sol} and \eqref{sub_sol} satisfies (a)-(c). Therefore, Propositions \ref{prop1} and \ref{monoprop} imply that the unique positive solution $\uep(x)$ of \eqref{inoue-eq} satisfies $\underline{\uep}(x) \leq \uep(x) \leq \overline{\uep}(x)$ for all $x \in \overline{B^n_1}$. In view of \eqref{vm_radial}, one can see that
\begin{align*}
  \|\underline{\uep}\|_{L^1(B^n_1)} &= A_n \int_0^1 \vep(r) r^{n-1} \,dr \\
  &=c_2 A_n \biggl( \int_0^\e \frac{r^{n-1}}{\e^n} e^{-r^n/\e^n} \,dr + \int_\e^1 \frac{1}{er} \,dr \biggr)\\
  &= c_2 A_n \biggl( \frac{1}{n} \biggl(1-\frac{1}{e}\biggr) + \frac{1}{e} |\log\e| \biggr),
\end{align*}
where $A_n$ denotes the surface area of $\partial B^n_1$. Since $\|\mep\|_{L^1(B_0(1))}=|B^n_1|=A_n/n$, then we have
\begin{equation}\label{belowest}
 \frac{ \|\uep\|_{L^1(B^n_1)} }{ \|\mep\|_{L^1(B^n_1)} } \ge \frac{ \|\underline{\uep}\|_{L^1(B^n_1)} }{ \|\mep\|_{L^1(B^n_1)} } = c_2 \biggl( 1-\frac{1}{e} + \frac{n}{e}\,|\log \e| \biggr). 
\end{equation}
Thus the proof of Theorem \ref{mainthm2} is complete.
\end{proof}

\begin{proof}[Proof of Theorem \ref{main_thm}]
By setting $\e \to 0$ in \eqref{belowest}, we see that the unique positive solution $\uep(x)$ of \eqref{inoue-eq} satisfies
\[ \lim_{\e \to 0} \frac{\|\uep\|_{L^1(B^n_1)}}{\|\mep\|_{L^1(B^n_1)}} = \infty, \]
which implies \eqref{L1ratio_infty}. The proof of Theorem \ref{main_thm} is complete.
\end{proof}

\section{Concluding remarks}

In this section, we give some concluding remarks.
For each dimension number $n\ge 1$ and any
$(d,m)\in (0,\infty) \times L^{\infty}_{+}(B^{n}_{1})$,
we define
\[
I_{n}(d,m):=\dfrac{\|u_{d,m}\|_{L^{1}(B^{n}_{1})}}{\|m\|_{L^{1}(B^{n}_{1})}},\]
where
$B^{n}_{1}:=\{\,x\in\mathbb{R}^{n}\,|\,|x|<1\,\}$.
Concerning the maximizing problem to consider
\[
M_{n}:=\sup_{(d,m) \in (0,\infty) \times L^{\infty}_{+}(B^{n}_{1}) }I_{n}(d,m),
\]
Theorem \ref{main_thm} reveals a fact that $M_{n}=\infty$ if $n\ge 2$,
which is a big contrast to the one-dimensional situation 
$M_{1}=3$ obtained by \cite{BaiHeLi}.

In the one-dimensional case when $n=1$,
a maximizing sequence 
$(d_{\e}, m_{\e})=(\sqrt{\e}, \e^{-1}\chi_{[0,\e]})$
realizes $I(d_{\e}, m_{\e})\nearrow M_{1}=3$
as $\e\searrow 0$ (\cite{BaiHeLi}),
where $\chi_{A}$ denotes the characteristic function of the set $A$.
This result says that, 
under the concentration setting of resource as $m_{\e}=\e^{-1}\chi_{[0,\e]}$,
a small control of the diffusion rate as $d_{\e}=\sqrt{\e}$
can make $I_{1}(d_{\e}, m_{\e})$ tend to the supremum $3$ from below
as $\e\to +0$.
In this situation as $\e\to +0$,
the profile of $u_{\e}(x)$ obtained by \cite{Inoue} shows that
$u_{\e}(x)$ with $x\in[0,\e]$ grows with the order $O(1/\sqrt{\e})$.
This result means that the species in the resource interval $[0,\e]$
cannot follow the height (the $L^{\infty}$ norm) $1/\e$ of resource.
Furthermore, the singular limit $d_{\e}=\sqrt{\e}\to 0$
leads to the shrink property that
$u_{\e}(x)\to 0$ uniformly in any compact set contained
in $(0,1]$ as $\e\to +0$.

In the two-dimensional case when $n= 2$,
Theorem \ref{mainthm2} asserts that,
under the concentration of resource near the center
as $m_{\e}=\e^{-2}\chi_{B^{2}_{\e}}$,
a middle control of the diffusion rate as $d_{\e}=c_{1}$
(independent of $\e$) can make $I_{2}(c_{1}, m_{\e})$
tend to infinity as $\e\to +0$.
Furthermore, the profile of the sub-solution 
$\underline{u_{\e}}=c_{2}\e^{-2}\exp(-(|x|/\e)^{2})$ 
for $x\in \overline{B^{2}_{\e}}$ 
ensures that
$u_{\e}(x)$
$(\,\ge \underline{u}_{\e}(x)\,)$
for $x\in \overline{B^{2}_{\e}}$ grows with the order
$O(1/\e^{2})$ as $\e\to +0$.
This fact means that the species in the resource disk
$B^{2}_{\e}$ can follow the height (the $L^\infty$ norm)
$1/\e^{2}$ of resource.
On the other hand, in the no-resource annulus
$B^{2}_{1}\setminus B^{2}_{\e}$,
the sub-solution $\underline{u}_{\e}(x)=c_{2}e^{-1}|x|^{-2}$
for $x\in \overline{B^{2}_{1}}\setminus \overline{B^{2}_{\e}}$
implies that $u_{\e}(x)\ge \underline{u}_{\e}(x)$ for 
$x\in \overline{B^{2}_{1}}\setminus \overline{B^{2}_{\e}}$.
This fact is also a big difference from the one-dimensional case
that $u_{\e}(x)$ decays to zero in any compact set contained in
$(0,1]$ as $\e\to +0$.

In the higher dimensional case when $n\ge 3$,
under the concentration of resource as $m_{\e}=\e^{-n}\chi_{B^{n}_{\e}}$,
a large control of the diffusion rate as $d_{\e}=c_{1}/\e^{n-2}$
can make $I_{n}(d_{\e}, m_{\e})$ tend to infinity as $\e\to +0$.
In this situation as $\e\to +0$, the profile of
the sub-solution $\underline{u}_{\e}(x)$ tells us that
$u_{\e}(x)$ can follow $m_{\e}(x)$ in the resource ball $\overline{B^{n}_{\e}}$
with the same order $O(1/\e^{n})$ and
$u_{\e}(x)\ge \underline{u}_{\e}(x)=c_{2}\e^{-1}|x|^{-n}$
for the no-resource region $\overline{B^{n}_{1}}\setminus \overline{B^{n}_{\e}}$.



\medskip
Received xxxx 20xx; revised xxxx 20xx.
\medskip

\end{document}